\documentclass{amsart}
\usepackage{amsfonts,amssymb,amsmath,amsthm}
\usepackage{url}
\usepackage{enumerate}

\usepackage{amsmath}
\usepackage{amssymb}
\usepackage{tikz}

\renewcommand{\geq}{\geqslant}
\renewcommand{\leq}{\leqslant}

\newtheorem{thm}{Theorem}[section]
\newtheorem{lemma}[thm]{Lemma}
\newtheorem{cor}[thm]{Corollary}
\newtheorem{definition}[thm]{Definition}

\title{Connectivity in Hypergraphs}

\author{Megan Dewar}
\address{Tutte Institute for Mathematics and Computing\\ Ottawa, ON, Canada}
\email{tutte.institute+MeganDewar@gmail.com}

\author{David Pike}
\address{Department of Mathematics and Statistics\\ Memorial University of Newfoundland\\ St.~John's, NL, Canada}
\email{dapike@mun.ca}
\thanks{D.~Pike acknowledges research support from NSERC (grant number RGPIN-2016-04456).}

\author{John Proos}
\address{Tutte Institute for Mathematics and Computing\\ Ottawa, ON, Canada}
\email{tutte.institute+JohnProos@gmail.com}


\begin{document}
\maketitle

\begin{abstract}
In this paper we consider two natural notions of connectivity for hypergraphs:  weak and strong.  We prove that the strong vertex connectivity of a connected hypergraph is bounded by its weak edge connectivity, thereby extending a theorem of Whitney from graphs to hypergraphs.  We find that while determining a minimum weak vertex cut can be done in polynomial time and is equivalent to finding a minimum vertex cut in the 2-section of the hypergraph in question, determining a minimum strong vertex cut is NP-hard for general hypergraphs.  Moreover, the problem of finding minimum strong vertex cuts remains NP-hard when restricted to hypergraphs with maximum edge size at most 3.  We also discuss the relationship between strong vertex connectivity and the minimum transversal problem for hypergraphs, showing that there are classes of hypergraphs for which one of the problems is NP-hard while the other can be solved in polynomial time.
\end{abstract}



\section{Introduction}
\label{intro}

When extending concepts from graph theory to the realm of hypergraph theory it is not unusual for there to be multiple natural ways in which the concepts can be generalized.  This is certainly the case when considering aspects of connectivity.  For a nontrivial graph $G$, its connectivity $\kappa(G)$ is defined as the least number of vertices whose deletion from $G$ results in a graph that is not connected.  Since each edge of $G$ is a 2-subset of the vertex set of the graph, deleting a vertex from an edge also results in the removal of that edge from the edge set of the graph.  However, for hypergraphs the removal of a vertex from each edge that contains it may have quite a different effect from also removing these edges from the edge set.  These distinctions give rise to the notions of {weak vertex deletion} and {strong vertex deletion}, respectively.  We correspondingly define the {weak vertex connectivity} $\kappa_W(H)$ (resp., the {strong vertex connectivity} $\kappa_S(H)$) of a nontrivial hypergraph $H$ to be the least number of vertices whose weak (resp., strong) deletion from $H$ results in a disconnected hypergraph, thereby generalizing connectivity as usually defined for graphs (see~\cite{BondyMurty2008} for more details about graph terminology).

Hypergraphs with weak vertex connectivity $\kappa_W(H)=1$ have recently been considered by Bahmanian and \v{S}ajna~\cite{bahmanian}.  In the present paper we consider vertex cuts of any size, with emphasis on strong vertex connectivity.  We also introduce weak and strong edge connectivity for hypergraphs.

In 1932, Whitney established one of the basic results on connectivity for graphs:  the connectivity $\kappa(G)$ of a nontrivial graph $G$ is bounded above by the edge-connectivity $\kappa'(G)$, which in turn is bounded above by the minimum degree $\delta(G)$ of the graph $G$~\cite{whitney}.  We extend this result to hypergraphs and, in so doing, also introduce notions of weak and strong edge deletion. Here, weak deletion of an edge $e$ merely involves removing $e$ from the edge multiset of the hypergraph.  By defining $\kappa'_W(H)$ to be the least number of edges whose weak deletion from $H$ results in a disconnected hypergraph, we  show that Whitney's result can be generalized as follows:

{
\renewcommand{\thethm}{\ref{thm:WhitneyHypergraph}}
\begin{thm}
Let $H=(V,E)$ be a nontrivial hypergraph with minimum degree $\delta(H)$.
Then $\kappa_S(H) \leq \kappa'_W(H) \leq \delta(H)$.
\end{thm}
\addtocounter{thm}{-1}
}

Maximum-flow minimum-cut algorithms
that have polynomially bounded running times,
such as the Edmonds-Karp algorithm~\cite{EdmondsKarp},
can be used to efficiently
compute the connectivity $\kappa(G)$ of any given graph $G$, as well as to find a corresponding vertex cut of cardinality $\kappa(G)$
(a polynomial time algorithm for determining the connectivity of a graph $G$ can also be found in~\cite[page~42]{NagamochiIbaraki}).
Hence these problems are in the class P of problems that can be solved in polynomial time.
We show that this is also the case for weak vertex connectivity of hypergraphs:

{
\renewcommand{\thethm}{\ref{thm_kappa_w_poly}}
\begin{thm}
Determining $\kappa_W(H)$ for a hypergraph $H$ and finding a minimum weak vertex cut of $H$ are in P.
\end{thm}
\addtocounter{thm}{-1}
}

However, determining the strong connectivity of a hypergraph is, in general, computationally intractable, which is in stark contrast to graphs (for which ``strong'' connectivity coincides with ``weak'' connectivity and hence can be determined in polynomial time):


{
\renewcommand{\thethm}{\ref{thm_kappa_strong}}
\begin{thm}
The problem of determining $\kappa_S(H)$ is NP-hard for arbitrary hypergraphs. Furthermore, the problem remains NP-hard when $H$ is restricted to hypergraphs with maximum edge size at most 3.
\end{thm}
\addtocounter{thm}{-1}
}

We subsequently consider the complexity of calculating $\kappa_S(H)$ for various classes of hypergraphs, as well as the problem of determining the size $\tau(H)$ of a minimum transversal of a hypergraph $H$ (which is also NP-hard when considered over the set of all hypergraphs
since it is equivalent to the Set Covering problem of~\cite{karp}).
We show that despite similarities between transversals and strong vertex cuts,
there exist classes of hypergraphs for which calculating $\kappa_S(H)$ is NP-hard while calculating $\tau(H)$ is in P, and vice-versa.

\section{Background terminology and notation}
\label{definitions}

In this section we introduce notation and several concepts that are necessary when studying connectivity in hypergraphs.  We begin with a review of basic hypergraph terminology.

\subsection{Basic definitions}
A \textbf{hypergraph} $H$, denoted $H = (V, E)$, consists of a set $V$ of \textbf{vertices} together with a multiset $E = (e_i)_{i \in I}$ of submultisets of $V$ called \textbf{edges};
$E$ is indexed by an \textbf{index set} $I$.
Throughout this paper we consider only finite hypergraphs, i.e., $V$ and $I$ are both finite.

A hypergraph with no vertices is called a \textbf{null} hypergraph, a hypergraph with only one vertex is called a \textbf{trivial} hypergraph,
and all other hypergraphs are \textbf{nontrivial}.
A hypergraph with no edges is called \textbf{empty}.  Given that $E$ is a multiset, a hypergraph may contain repeated edges.
If $e_i = e_j$, then $e_i$ and $e_j$ are said to be \textbf{parallel}.  The number of edges parallel to $e_i$ (including $e_i$) is the \textbf{multiplicity} of $e_i$.

For $v,w \in V$, $v$ and $w$ are said to be \textbf{adjacent} if there exists an edge $e_i \in E$ such that the multiset $[v, w] \subseteq e_i$ (we use square braces when listing elements of a multiset, and we use curly brackets for sets).  We explicitly allow $v = w$, and use the notation $[v; t]$ to denote the multiset consisting of $t$ copies of the element $v$.  A vertex $v$ and an edge $e_i$ are said to be \textbf{incident} if $v \in e_i$.
Let $m_{e_i}^H(v)$ denote the \textbf{multiplicity} with which $v$ appears in edge $e_i$ of hypergraph $H$.  The \textbf{degree} of vertex $v$ in $H$ is defined as $deg_H(v) = \sum_{i \in I} m_{e_i}^H(v)$;
if the context is unambiguous, then we may simply write $deg(v)$.
The \textbf{minimum degree} of $H$, denoted by $\delta(H)$, is $\min_{v \in V} deg(v)$.

We note that the definition of edges as submultisets of $V$ requires terminology to distinguish between the number of vertices in an edge (counting multiplicities) and the number of {distinct} vertices in an edge.  Let $|e_i|$ denote the number of elements in the multiset $e_i$, which we will refer to as the \textbf{size} of the edge $e_i$.  Taking terminology from multiset theory, let $supp(e_i)$, called the \textbf{support} of $e_i$, denote the set of distinct elements of the edge $e_i$, i.e., $supp(e_i) = \{v \in e_i \,|\, m_{e_i}^H (v) > 0\}$.
Multiset theory refers to $|supp(e_i)|$ as the \textbf{cardinality} of the multiset (edge) $e_i$.
A hypergraph is called \textbf{$\mathbf{k}$-uniform} if $|e_i| = k$ for all $i \in I$.  A hypergraph which is 2-uniform
is called a \textbf{graph} (or multigraph) and if, in addition, there are no parallel edges or loops, then it is called a \textbf{simple} graph.
A hypergraph $H=(V,E)$ is called \textbf{simple} if it has no parallel edges and each edge is a set
(i.e., $e_i = supp(e_i)$ for each edge $e_i \in E$).

In Section \ref{strong_v_con_complexity} we will be discussing matchings.  A \textbf{matching} in a hypergraph is a set of pairwise disjoint nonempty edges.

\subsection{Associated hypergraphs and graphs}

Given a hypergraph $H = (V, E)$, the \textbf{dual hypergraph}, denoted $H' = (V', E')$, is the hypergraph with $V' = I$ and $E' = (e'_v)_{v \in V}$, where $e'_v$ is the submultiset of $I$ with $m_{e'_v}^{H'}(i) = m_{e_i}^H(v)$.  That is, $H'$ has a vertex for every edge of $H$, an edge for every vertex of $H$, and the multiplicity of $i$ in the edge of $H'$ corresponding to the vertex $v$ of $H$ is the multiplicity in $H$ of $v$ in $e_i$.  It is perhaps easier to visualize the dualization process as creating the hypergraph whose incidence matrix is the transpose of the incidence matrix of the original.  The incidence matrix of a hypergraph $H$ is 
the $|V| \times |I|$ matrix
$M = (m_{ij})$, where rows are indexed by vertices,  columns are indexed by edges, and $m_{ij} = m_{e_j}^H(v_i)$ (i.e., $m_{ij}$ is the multiplicity of $v_i$ in $e_j$).


Given a hypergraph $H=(V,E)$, the \textbf{2-section} of $H$ is the graph denoted $[H]_2 = (V, E_2)$, where $[v, w] \in E_2$ if there exists an edge $e_i \in E$ such that $[v, w] \subseteq e_i$. Note that $[H]_2$ can contain loops, but not parallel edges.  The \textbf{incidence graph} of $H$ is the graph denoted $G(H) = (V \cup E, E')$,
where $E' = \big[    \{v, e_i\}; m_{e_i}^H(v)  \,|\, v \in V, i \in I,  v \in e_i    \big]$.
Note that $G(H)$ is bipartite with bipartition $(V, E)$.  An important observation is that the incidence graph retains complete information about the hypergraph, whereas the 2-section does not.

\subsection{Substructures of hypergraphs}

An extensive list of hypergraph constructions appears in Bahmanian and {\v S}ajna \cite{bahmanian}.  Some, but not all, of their terminology coincides with \cite{duchet}, while the terminology is somewhat reversed in \cite{voloshin}.  Here we introduce new terminology which we think more clearly represents the substructures created, while referencing the alternate naming conventions that have been employed in other presentations.

A hypergraph $H' = (V', E')$ is a \textbf{weak subhypergraph} of $H = (V, E)$ if $V' \subseteq V$, $I' \subseteq I$, $E' = (e'_i)_{i \in I'}$, and for each $i \in I'$, $e'_i = [v; m_{e_i}^H(v) \,|\, v \in e_i \cap V']$.
Equivalently, the incidence matrix of $H'$, after a suitable permutation of its rows and columns, is a submatrix of the incidence matrix of $H$.
Weak subhypergraphs are called ``subhypergraphs'' by Bahmanian and {\v S}ajna~\cite{bahmanian}, as well as by Duchet~\cite{duchet}.
A hypergraph $H'=(V',E')$ is call an \textbf{induced weak subhypergraph} of the hypergraph $H=(V,E)$ if $V' \subseteq V$ and $E' = (e'_i)_{i \in I'}$,
where $I' = \{i \in I \,|\, e_i \cap V' \ne \emptyset\}$ and $e'_i = [v; m_{e_i}^H(v) \,|\, v \in e_i \cap V]$ for each $i \in I'$.

A hypergraph $H' = (V',E')$ is called a \textbf{strong subhypergraph} of the hypergraph $H=(V,E)$
if $V' \subseteq V$ and $E' \subseteq E$.
These are termed ``hypersubgraphs'' by Bahmanian and {\v S}ajna~\cite{bahmanian}, and ``{partial (sub)hypergraphs}'' by Duchet~\cite{duchet}.
A strong subhypergraph $H'=(V',E')$ of $H=(V,E)$, with $E' = (e_i)_{i \in I'}$, is said to be
\textbf{induced by $V'$} if $I'=\{ i \in I \,|\, supp(e_i) \subseteq V' \}$,
and is said to be \textbf{induced by $E'$ (or $I'$)} if $V' = \bigcup_{i \in I'} supp(e_i)$.

Note that by definition every strong subhypergraph is also a weak subhypergraph. This is similar to connectivity of directed graphs, where every strongly connected directed graph is also weakly connected.

A subhypergraph $H'=(V',E')$ of a hypergraph $H=(V,E)$ is said to be a \textbf{spanning} subhypergraph if $V' = V$.

\subsection{Paths and walks}

Given a hypergraph $H = (V, E)$ we define a \textbf{walk} in $H$ to be an alternating sequence $v_1, e_1, v_2, \ldots, e_s, v_{s+1}$ of vertices and edges of $H$ such that:
\begin{enumerate}
\item $v_j \in V$ for $j = 1, \ldots, s+1$;
\item $e_j \in E$ for $j = 1, \ldots, s$; and
\item the multiset $[ v_j, v_{j+1} ] \subseteq e_j$ for $j = 1, \ldots, s$.
\end{enumerate}
Such a walk from $v_1$ to $v_{s+1}$ will be referred to as a \textbf{$\mathbf{(v_1,v_{s+1})}$-walk}. When the edges of a walk are either canonical or unimportant, we shall sometimes denote it simply as a sequence of vertices.
As in graph theoretic terminology, a \textbf{path} is a walk with the additional restrictions
that the $s+1$ vertices are all distinct and the $s$ edges are all distinct.
A \textbf{cycle} is a walk with $s$ distinct edges and $s$ distinct vertices such that $v_1 = v_{s+1}$.
The \textbf{length} of a walk, path or cycle is the number of edges (counting multiplicity for walks) in the sequence; i.e., it is $s$ in the foregoing definitions.
See \cite{bahmanian} for a more rigorous treatment of walks, paths, cycles and trails in the case of hypergraphs whose edges are sets.

Two vertices $v,w \in V$ are said to be \textbf{connected} in $H$ if there exists a $(v,w)$-path in $H$;
otherwise $v$ and $w$ are \textbf{separated} from each other.
A hypergraph $H$ is \textbf{connected} if every pair of vertices $v,w \in V$ is connected in $H$;
otherwise $H$ is \textbf{disconnected}.
A \textbf{connected component} of a hypergraph $H$ is a maximal connected weak subhypergraph of $H$. Note that the maximality condition implies that connected components will, in fact, be strong subhypergraphs.
We will use $c(H)$ to denote the number of connected components of $H$.

\subsection{Vertex and edge deletion}

Given a hypergraph $H = (V, E)$ we can form new hypergraphs by deleting vertices in the following ways:

\begin{itemize}
\item \textbf{strong vertex deletion} of a vertex $v \in V$ creates the hypergraph $H' = (V', E')$ where $V' = V \setminus \{v\}$, $E'= (e_i)_{i \in I'}$ and $I' = \{i \in I \,|\, v \notin e_i\}$. That is, strong deletion of $v$ removes $v$ and all edges that are incident to $v$ from the hypergraph.
    We use the notation $H \setminus_{S} v$ to denote the hypergraph formed by strongly deleting the vertex $v$ from $H$.
    For any subset $X$ of $V$, we use $H \setminus_{S} X$ to denote the hypergraph formed by strongly deleting all the vertices of $X$ from $H$.

\item \textbf{weak vertex deletion} of a vertex $v \in V$ creates the hypergraph $H' = (V', E')$ where $V' = V \setminus \{v\}$ and $E' = (e'_i)_{i \in I}$ such that for $i \in I$ we have
    $e'_i = e_i \setminus [v; m_{e_i}^H(v)]$. That is, weak deletion of $v$ removes $v$ from the vertex set, and all occurences of $v$ from the edges of the hypergraph $H$.
    We use the notation $H \setminus_{W} v$ to denote the hypergraph formed by weakly deleting the vertex $v$ from $H$.
    For any subset $X$ of $V$, we use $H \setminus_{W} X$ to denote the hypergraph formed by weakly deleting all the vertices of $X$ from $H$.
\end{itemize}
Similarly, we define strong and weak edge deletion as follows:
\begin{itemize}
\item \textbf{strong edge deletion} of an edge $e_j \in E$ creates the hypergraph $H' = (V', E')$ where $V' = V \setminus supp(e_j)$, $E' = (e'_i)_{i \in I'}$, $I' = I \setminus \{j\}$ and $e'_i = e_i \setminus [v; m_{e_i}^H(v) \,|\, v \in e_j]$ for $i \ne j$. That is, strong edge deletion of $e_j$ removes $e_j$ from the hypergraph and weakly deletes all the vertices incident with $e_j$.
    We use the notation $H \setminus_{S} e_j$ to denote the hypergraph formed by strongly deleting the edge $e_j$ from $H$.
    For any submultiset $F$ of $E$, we use $H \setminus_{S} F$ to denote the hypergraph formed by strongly deleting all the edges of $F$ from $H$.
\item \textbf{weak edge deletion} of an edge $e_j \in E$ creates the hypergraph $H' = (V, E')$ where $I' = I \setminus \{j\}$ and $E' = (e_i)_{i \in I'}$. That is, weak edge deletion of $e_j$ simply removes $e_j$ without affecting the rest of the hypergraph.
    We use the notation $H \setminus_{W} e_j$ to denote the hypergraph formed by weakly deleting the edge $e_j$ from $H$.
    For any submultiset $F$ of $E$, we use $H \setminus_{W} F$ to denote the hypergraph formed by weakly deleting all the edges of $F$ from $H$.
\end{itemize}

Note that strong and weak edge deletion in a hypergraph $H$ correspond to strong and weak vertex deletion, respectively, in the dual of $H$.

It is also worth noting that every weak subhypergraph of a hypergraph $H$ can be formed by performing a sequence of weak vertex and weak edge deletions on $H$. Likewise, every strong subhypergraph of $H$ can be formed by performing a sequence of strong vertex and weak edge deletions on $H$.

\section{Connectivity}
For a graph $G$, a \textbf{cut vertex} is any vertex whose deletion from $G$ increases the number of connected components, while a \textbf{vertex cut} is any set of vertices whose deletion from $G$ results in a disconnected graph.
We generalize these concepts from graphs to hypergraphs.
\begin{definition}
Let $H = (V,E)$ be a nontrivial hypergraph.  A vertex $v \in V$ is called a \textbf{weak cut vertex} of $H$ if $H \setminus_{W} v$ has more connected components than $H$, and a set $X \subseteq V$ is called a \textbf{weak vertex cut} of $H$ if $H \setminus_{W} X$ is disconnected.  We define the \textbf{weak vertex connectivity} of $H$, denoted $\kappa_W(H)$, as follows:  if $H$ has at least one weak vertex cut, then $\kappa_W(H)$ is the cardinality of a minimum weak vertex cut of $H$; otherwise, $\kappa_W(H)=|V|-1$.  Similar to the convention used in~\cite[page~207]{BondyMurty2008} for connectivity of trivial graphs, we adopt the convention that the weak vertex connectivity of a null or trivial hypergraph is 1.
\end{definition}

\begin{definition}
Let $H = (V,E)$ be a nontrivial hypergraph. A vertex $v \in V$ is called a \textbf{strong cut vertex} of $H$ if $H \setminus_{S} v$ has more connected components than $H$, and a set $X \subseteq V$ is called a \textbf{strong vertex cut} of $H$ if $H \setminus_{S} X$ is disconnected. We define the \textbf{strong vertex connectivity} of $H$, denoted $\kappa_S(H)$, as follows:  if $H$ has at least one strong vertex cut, then $\kappa_S(H)$ is the cardinality of a minimum strong vertex cut of $H$; otherwise, $\kappa_S(H)=|V|-1$. By convention, the strong vertex connectivity of a null or trivial hypergraph is 1.
\end{definition}

Note that for connected hypergraphs weak (resp.\ strong) cut vertices correspond to weak (resp.\ strong) vertex cuts of size 1. The convention that null and trivial hypergraphs have weak and strong vertex connectivity 1 ensures that a hypergraph $H$ has $\kappa_W(H)=0$ (or $\kappa_S(H) =0$) if and only if $H$ is disconnected. This generalizes vertex connectivity of graphs as defined in~\cite{BondyMurty2008}.

Observe that neither weak vertex connectivity nor strong vertex connectivity is affected
by edges with multiplicity exceeding 1,
by edges of cardinality less than 2,
or by vertices having multiplicity (within an edge) exceeding 1. As such, unless otherwise stated, our discussion of vertex connectivity shall assume that the hypergraphs being considered have no repeated edges, no repeated vertices within an edge and no edges of size less than 2.

\begin{lemma}
Let $H=(V,E)$ be a hypergraph.
Then $\kappa_S(H) \leq \kappa_W(H)$.
\end{lemma}

\begin{proof}
If $H$ is a null or trivial hypergraph then $\kappa_S(H) = \kappa_W(H)=1$. If $H$ is disconnected then $\kappa_S(H) = \kappa_W(H)=0$.

Suppose now that $H$ is nontrivial and connected.
For any $X \subseteq V$, the hypergraph $H \setminus_S X$ is a spanning strong subhypergraph of $H \setminus_W X$.
Thus any two vertices that are adjacent in $H \setminus_S X$ are also adjacent in $H \setminus_W X$.
\end{proof}

In \cite{bahmanian} Bahmanian and {\v S}ajna used the phrase ``cut vertex''
for any vertex whose weak deletion increases the number of connected components of the hypergraph;
hence a ``cut vertex'' in the sense of Bahmanian and {\v S}ajna corresponds to a weak cut vertex.
It is proved in their Theorem~3.23 that if a hypergraph $H$ has no edges of size less than~2 then a vertex of $H$ is a (weak) cut vertex of $H$ if and only if it is a cut vertex (in the traditional sense for graphs) of the incidence graph of $H$.
Since graph connectivity can be determined in polynomial time
(see~\cite[page~42]{NagamochiIbaraki} for an algorithm), this provides a polynomial time means of identifying whether a connected hypergraph has any weak cut vertices.

We shall now introduce the notions of weak and strong edge connectivity for hypergraphs.

\begin{definition}\label{Defn-WeakEdgeConnectivity}
Let $H=(V,E)$ be a hypergraph.  A submultiset $F \subseteq E$ is called a \textbf{weak disconnecting set} of $H$ if $H \setminus_W F$ is disconnected.
If $H$ is nontrivial, then we define its \textbf{weak edge connectivity}, denoted $\kappa'_W(H)$, as the minimum cardinality of any of its weak disconnecting sets.
Similar to the convention used in~\cite[page~216]{BondyMurty2008} for edge connectivity of trivial graphs, we adopt the convention that the weak edge connectivity of a null or trivial hypergraph is 1.
\end{definition}

\begin{definition}
Let $H=(V,E)$ be a hypergraph.  A submultiset $F \subseteq E$ is called a \textbf{strong disconnecting set} of $H$ if $H \setminus_S F$ is disconnected.
If $H$ is nontrivial and has at least one strong disconnecting set,
then we define its \textbf{strong edge connectivity}, denoted $\kappa'_S(H)$, as the minimum cardinality of any of its strong disconnecting sets. If $H$ is nontrivial and has no strong disconnecting set then we define $\kappa'_S(H)=|E|$.
By convention, the strong edge connectivity of a null or trivial hypergraph is 1.
\end{definition}

Suppose $H=(V,E)$ is a connected hypergraph and that $X \subset V$. Define the \textbf{weak edge cut} of $H$ associated to $X$ as $\partial(X) = \big\{ e \in E \,\big|\, e \cap X \neq \emptyset, e \cap (V \setminus X ) \neq \emptyset \big\}$, that is, the submultiset of edges that are incident to at least one vertex in each of $X$ and $V \setminus X$. The size of a weak edge cut is the number of edges that it contains. Note that all weak edge cuts are disconnecting sets and that, as remarked by Cheng~\cite{Cheng1999}, all minimal weak disconnecting sets are weak edge cuts.
Thus $\kappa'_W(H)$ for nontrivial $H$ can equally well be defined as the size of a minimum weak edge cut.

In~\cite{Cheng1999} weak edge cuts are considered by Cheng who uses the phrase ``$k$-edge-connected'' to describe any hypergraph $H=(V,E)$ for which each weak edge cut of $H$ has at least $k$ edges.
Relating Cheng's terminology to the notation of Definition~\ref{Defn-WeakEdgeConnectivity}, it follows that $H$ is weakly $k$-edge-connected if and only if $k \leq \kappa'_W(H)$.
In~\cite{Frank2006}, using similar terminology to Cheng,
Frank observes that a hypergraph is weakly $k$-edge-connected if and only if there are $k$ edge-disjoint $(u,v)$-paths in $H$ for each pair of distinct vertices $u,v \in V$.
Chekuri and Xu ~\cite{ChekuriXu2017} have recently demonstrated an $O(p+n^2 \kappa'_W(H))$ time algorithm to find a minimum weak edge cut of a hypergraph $H$,
where $p = \sum_{e \in E} |e|$ and $n=|V|$.

Whereas $\kappa_W(H)$ (resp.\ $\kappa_S(H)$) represents the fewest number of vertices whose weak (resp.\ strong) deletion from a nontrivial hypergraph $H$
results in the separation of some pair of vertices,
on occasion we will want to separate specific vertices from each other.  We therefore introduce the following definitions and notation.

\begin{definition}
Let $H = (V,E)$ be a hypergraph and let $u,v \in V$, $u \neq v$.
A set $X \subseteq V \setminus \{u,v\}$ is called a \textbf{weak $\mathbf{(u,v)}$-vertex cut} (resp.\ \textbf{strong $\mathbf{(u,v)}$-vertex cut}) in $H$
if $u$ and $v$ are separated in $H \setminus_{W} X$ (resp.\ $H \setminus_{S} X$).
If $A_u$ and $A_v$ are disjoint subsets of $V$, then
a set $X \subseteq V \setminus (A_u \cup A_v)$ is called a \textbf{weak $\mathbf{(A_u,A_v)}$-vertex cut} (resp.\ \textbf{strong $\mathbf{(A_u,A_v)}$-vertex cut}) in $H$
if, for each $u \in A_u$ and each $v \in A_v$, $u$ and $v$ are separated in $H \setminus_{W} X$ (resp.\ $H \setminus_{S} X$).
We denote the cardinality of a minimum weak (resp.\ strong) $(u,v)$-vertex cut in $H$ by $\kappa_W(H,u,v)$ (resp.\ $\kappa_S(H,u,v)$);
if $u$ and $v$ cannot be separated by any weak (resp.\ strong) vertex cut, then we set $\kappa_W(H,u,v) = |V|-1$ (resp.\ $\kappa_S(H,u,v) = |V|-1$).
\end{definition}

For edge deletion we have similar definitions.

\begin{definition}
Let $H = (V,E)$ be a hypergraph and let $u,v \in V$, $u \neq v$.
A set $F \subseteq E$ is called a \textbf{weak $\mathbf{(u,v)}$-disconnecting set} (resp.\ \textbf{strong $\mathbf{(u,v)}$-disconnecting set}) in $H$
if $u$ and $v$ are separated in $H \setminus_{W} F$ (resp.\ $H \setminus_{S} F$).
If $A_u$ and $A_v$ are disjoint subsets of $V$, then
a set $F \subseteq E$ is called a \textbf{weak $\mathbf{(A_u,A_v)}$-disconnecting set} (resp.\ \textbf{strong $\mathbf{(A_u,A_v)}$-disconnecting set}) in $H$
if, for each $u \in A_u$ and each $v \in A_v$, $u$ and $v$ are separated in $H \setminus_{W} F$ (resp.\ $H \setminus_{S} F$).
We denote the cardinality of a minimum weak (resp.\ strong) $(u,v)$-disconnecting set in $H$ by $\kappa'_W(H,u,v)$ (resp.\ $\kappa'_S(H,u,v)$, with the convention that $\kappa'_S(H,u,v)=|E|$ when no strong $(u,v)$-disconnecting set exists).
\end{definition}

\subsection{Weak versus strong vertex connectivity}

In comparing weak vertex connectivity with strong vertex connectivity,
an important initial observation is
that it is false that every strong vertex cut of size $\kappa_S(H)$ is a subset of a weak vertex cut of size $\kappa_W(H)$.
In fact, there exists a hypergraph $H$ in which any minimum weak vertex cut and any minimum strong vertex cut are disjoint.
For example, take two copies of the complete graph on four vertices; one on the vertex set $\{x_1, x_2, x_3, x_4\}$ denoted $K_{4_x}$, and the other on vertex set $\{y_1, y_2, y_3, y_4\}$ denoted $K_{4_y}$. Define $H = (V,E)$ such that
\[V = \{x_1, x_2, x_3, x_4,y_1, y_2, y_3, y_4, z\}\]
and \[E = E(K_{4_x}) \cup E(K_{4_y}) \cup \big\{\{x_1,y_1, z\}, \{x_2,y_2,z\}\big\}\]
(see Figure~\ref{fig:disjoint_cuts}). Then $\{z\}$ is the unique minimum strong vertex cut of $H$, while the minimum weak vertex cuts of $H$ are $\{x_1, x_2\}$ and $\{y_1,y_2\}$.

\begin{figure}[tbhp]
\begin{center}

\begin{tikzpicture}[scale=.75, smooth cycle]
\draw [-, black, thick] (2,1) -- (2,-1) -- (4,1) -- (4,-1) -- (2,1) -- (4,1);
\draw [-, black, thick] (2,-1) -- (4,-1);

\draw [-, black, thick] (-2,1) -- (-2,-1) -- (-4,1) -- (-4,-1) -- (-2,1) -- (-4,1);
\draw [-, black, thick] (-2,-1) -- (-4,-1);

\draw [black,fill] (0,0) circle [radius=0.16] node [right] {~$z$};

\draw [black,fill] (2,1) circle [radius=0.16] node [left] {$y_1$ ~};
\draw [black,fill] (2,-1) circle [radius=0.16] node [left] {$y_2$ ~};
\draw [black,fill] (4,1) circle [radius=0.16] node [right] {~$y_3$};
\draw [black,fill] (4,-1) circle [radius=0.16] node [right] {~$y_4$};

\draw [black,fill] (-2,1) circle [radius=0.16] node [right] {~$x_1$};
\draw [black,fill] (-2,-1) circle [radius=0.16] node [right] {~$x_2$};
\draw [black,fill] (-4,1) circle [radius=0.16] node [left] {$x_3$ ~};
\draw [black,fill] (-4,-1) circle [radius=0.16] node [left] {$x_4$ ~};

\draw [black, thick] plot[tension=0.5] coordinates{(0,-0.5) (2.25,0.65) (2.25,1.35) (-2.25,1.35) (-2.25,0.65)};

\draw [black, thick] plot[tension=0.5] coordinates{(0,0.5) (2.25,-0.65) (2.25,-1.35) (-2.25,-1.35) (-2.25,-0.65)};
\end{tikzpicture}
\end{center}

\caption{Hypergraph with disjoint minimum weak and strong vertex cuts.}
\label{fig:disjoint_cuts}
\end{figure}
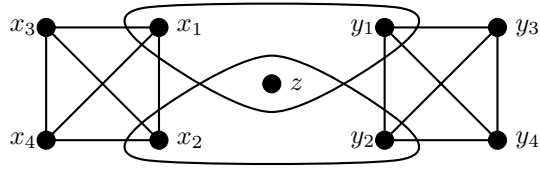

Also observe that the difference $\kappa_W(H) - \kappa_S(H)$ can be arbitrarily large.  As an example, let $n \geq 2$ and consider the hypergraph $H=(V,E)$ where $V=\{ x_1,\ldots,x_n,$ $y_1,\ldots,y_n,z \}$ and $E$ consists of the edges $\{x_1,\ldots,x_n\}$, $\{y_1,\ldots,y_n\}$ and $\{ x_i,y_i,z \}$ for each $i=1,2,\ldots,n$ (see Figure~\ref{fig:arb_diff_cuts}).
Here $\kappa_S(H)=1$ because $\{z\}$ is a strong vertex cut, whereas $\kappa_W(H)=n+1$
(observe that a minimum weak vertex cut is obtained by taking vertex $z$ together with either $x_i$ or $y_i$, for each $i$, so that at least one $x$-vertex and at least one $y$-vertex is selected).
This example illustrates that weak vertex connectivity is a poor approximation for strong vertex connectivity.

\begin{figure}[tbhp]
\begin{center}

\begin{tikzpicture}[scale=.75, smooth cycle]


\draw [black,fill] (0,0) circle [radius=0.16] node [right] {~$z$};

\draw [black,fill] (-2.5,1) circle [radius=0.16] node [below right] {~$x_1$ ~};
\draw [black,fill] (-2.5,0) circle [radius=0.16] node [right] {~$x_2$ ~};
\draw [black,fill] (-2.5,-0.6) circle [radius=0.05];
\draw [black,fill] (-2.5,-1) circle [radius=0.05];
\draw [black,fill] (-2.5,-1.4) circle [radius=0.05];
\draw [black,fill] (-2.5,-2) circle [radius=0.16] node [above right] {~$x_n$ ~};

\draw [black,fill] (2.5,1) circle [radius=0.16] node [below left] {~$y_1$ ~};
\draw [black,fill] (2.5,0) circle [radius=0.16] node [left] {~$y_2$ ~};
\draw [black,fill] (2.5,-0.6) circle [radius=0.05];
\draw [black,fill] (2.5,-1) circle [radius=0.05];
\draw [black,fill] (2.5,-1.4) circle [radius=0.05];
\draw [black,fill] (2.5,-2) circle [radius=0.16] node [above left] {~$y_n$ ~};

\draw [black, thick] plot[tension=0.5] coordinates{(1.6,1.25) (3,1.25) (3, -2.25) (1.6,-2.25)};
\draw [black, thick] plot[tension=0.5] coordinates{(-1.6,1.25) (-3,1.25) (-3, -2.25) (-1.6,-2.25)};

\draw [black, thick] plot[tension=0.5] coordinates{(0,-0.4) (2.75,0.75)  (2.5,1.5) (0,0.5) (-2.5,1.5)  (-2.75,0.75)};

\draw [black, thick] plot[tension=0.5] coordinates{(0,-0.3) (2.75, -0.2) (2.75,0.2)  (0,0.3)  (-2.75,0.2) (-2.75,-0.2)};

\draw [black, thick] plot[tension=0.5] coordinates{(0,0.4) (2.75,-1.85)  (2.1,-2.4) (0,-0.5) (-2.1,-2.4)   (-2.75,-1.85)};

\end{tikzpicture}
\end{center}

\caption{Hypergraph for which the difference between the weak vertex connectivity and strong vertex connectivity is $n$.
As $n$ increases, an infinite family of hypergraphs for which this difference grows linearly is obtained.}
\label{fig:arb_diff_cuts}
\end{figure}
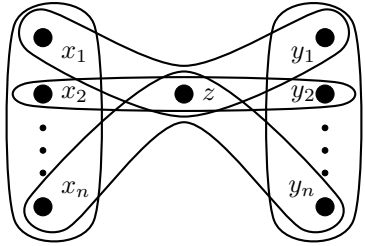

\subsection{Extending a result of Whitney}

In 1932 Whitney \cite{whitney} showed that for any nontrivial graph $G$ the size of a minimum vertex cut is at most the size of a minimum edge cut, which in turn is at most the minimum degree of any vertex, i.e., $\kappa(G) \leq \kappa'(G) \leq \delta(G)$.
In this section we will show that this result can be generalized to hypergraphs.
Note, however, that the generalization does not involve a comparison of $\kappa_W(G)$ with $\kappa'_W(G)$
since there exists a hypergraph $H$ such that $\kappa_W(H) > \kappa'_W(H)$ (see Figure~\ref{fig:KappaPrimeSmaller}),
as well as a hypergraph $H$ such that $\kappa_W(H) < \kappa'_W(H)$
(for instance, if $H=(V,E)$ where $V=\{x_1,x_2,x_3,y_1,y_2,y_3,z\}$ and
$E =
\big\{ \{x_i,x_j,z\} \,|\, 1 \leq i < j \leq 3 \big\} \cup
\big\{ \{y_i,y_j,z\} \,|\, 1 \leq i < j \leq 3 \big\}$,
then $1=\kappa_W(H)<\kappa'_W(H)=2$).
Likewise, the generalization does not bound $\kappa_S(H)$ by $\kappa'_S(H)$ since there exists a hypergraph $H$
such that $\kappa_S(H) < \kappa'_S(H)$ (for instance the hypergraph in Figure~\ref{fig:disjoint_cuts} without the edges $\{x_1,x_2\}$ and $\{y_1,y_2\}$),
as well as a hypergraph $H$ such that
$\kappa_S(H) > \kappa'_S(H)$
(for instance,  if $H=(V,E)$ with $V=\{0,1,2,3,4,5,6,3',4',5',6'\}$
and $E=\big\{
\{0,1,2\},$
$\{0,3,6\},$
$\{0,4,5\},$
$\{1,3,4\},$
$\{1,5,6\},$
$\{2,3,5\},$
$\{2,4,6\},$
$\{0,3',6'\},$
$\{0,4',5'\},$
$\{1,3',4'\},$
$\{1,5',6'\},$
$\{2,3',5'\},$
$\{2,4',6'\}
\big\}$,
then $\kappa'_S(H)=1$ since $\big\{ \{0,1,2\} \big\}$ is a strong disconnecting set, and $\kappa_S(H)=3$).
The generalization of Whitney's result that we establish involves strong vertex connectivity and weak edge connectivity.

\begin{figure}[htbp]

\begin{center}
\begin{tikzpicture}[scale=.75, smooth cycle]
\foreach \x in {0,2,4,6}
   {
   \draw [black,fill] (\x,0.3) circle [radius=0.16];
   \draw [black,fill] (\x,1.7) circle [radius=0.16];
   }
\foreach \x in {1,3,5}
   {
   \draw [black, thick] (\x,1) ellipse (48pt and 43pt);
   }

\end{tikzpicture}
\end{center}

\caption{Hypergraph $H$ with $2 = \kappa_W(H) > \kappa'_W(H) = 1$.}
\label{fig:KappaPrimeSmaller}
\end{figure}
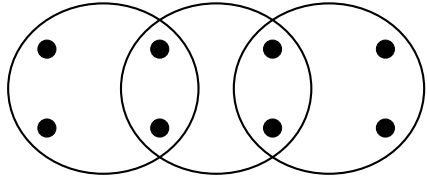

Since both vertex degrees and weak disconnecting sets are affected by repeated vertices in edges, repeated edges, and edges of cardinality one,
hypergraphs that have any or all of these will be permitted in Lemma~\ref{lemma_kappa_delta}
and Theorem~\ref{thm:WhitneyHypergraph}.

\begin{lemma} \label{lemma_kappa_delta}
Let $H=(V,E)$ be a nontrivial hypergraph with minimum degree $\delta(H)$.
Then $\kappa'_W(H) \leq \delta(H)$.
\end{lemma}

\begin{proof}
Let $v \in V$ be a vertex of degree $\delta(H)$.
Then $\big\{ e \in E \,\big|\, v \in e, |supp(e)| \geq 2 \big\}$ is a weak disconnecting set of size at most $\delta(H)$.
\end{proof}

\begin{thm} \label{thm:WhitneyHypergraph}
Let $H=(V,E)$ be a nontrivial hypergraph with minimum degree $\delta(H)$.
Then $\kappa_S(H) \leq \kappa'_W(H) \leq \delta(H)$.
\end{thm}

\begin{proof}
Since Lemma~\ref{lemma_kappa_delta} shows that $\kappa'_W(H) \leq \delta(H)$, it only remains to prove that $\kappa_S(H) \leq \kappa'_W(H)$.
Clearly $\kappa_S(H) \leq \kappa'_W(H)$ when $\kappa_S(H)=0$,
so we henceforth assume that $\kappa_S(H) \geq 1$, which is to say that $H$ is connected.

Let $F$ be a minimum weak disconnecting set of $H$. As $\kappa_S(H) \leq |V|-1$ by definition, if $|F| \geq |V| -1$, the result holds. Thus we shall now assume that $|F| < |V| -1$, and we will show the existence of a strong vertex cut of size at most $|F|$.

As $F$ is a weak disconnecting set, the hypergraph $H \setminus_W F $ is disconnected. Let $H_1,H_2,\dots , H_k$ be the connected components of $H \setminus_W F$. Furthermore, for $i = 1,2,\dots, k$, let $V_i$ be
the vertex set of $H_i$ (note that $\bigcup_{i=1}^{k} V_i = V$) and let $W_i$ be the subset of $V_i$ containing those vertices of $V_i$ that are
incident to at least one edge in $F$. Observe that $W_i \neq \emptyset$ for all $i$.

By the minimality of $F$ we know that every edge of $F$ intersects every $W_i$. This implies that for each $i$ there exists a set $Z_i \subseteq W_i$ such that $|Z_i| \leq |F|$ and $Z_i$ intersects every edge of $F$. The sets $Z_i$ can be found by a greedy approach:  start with $Z_i = \emptyset$ and while there exists an edge $f \in F$ that does not intersect $Z_i$ select any vertex in $f \cap W_i$ and add it to $Z_i$.

Suppose that $k \geq 3$. Then strongly deleting the vertices of $Z_3$ will delete all the edges of $F$ and separate the vertices in $V_1$ from those in $V_2$. Therefore when $k \geq 3$ there exists a strong vertex cut of size at most $|F|$.

Now assume that $k=2$
and
suppose that there exists a vertex $v \in V_1 \setminus Z_1$. Then $Z_1$ is a strong vertex cut that separates $v$ from $V_2$ and we again have a strong vertex cut of size at most $|F|$. Similarly, such a strong vertex cut also exists if there is a vertex $v \in V_2 \setminus Z_2$.

We are now in the case where $k=2$, $V_1 = W_1 = Z_1$ and $V_2 = W_2 = Z_2$.
Suppose that there exists a pair of vertices $v_1 \in V_1$ and $v_2 \in V_2$ such that $F$ does not contain an edge of cardinality two containing these two vertices. Then we can greedily find a set $Z \subseteq  (V_1 \setminus \{v_1\})\cup (V_2 \setminus \{v_2\}) = V \setminus \{v_1,v_2\}$ of size at most $|F|$ that intersects every edge of $F$:  start with $Z = \emptyset$ and while there exists an edge $f \in F$ that does not intersect $Z$, select any vertex in $f \cap (V \setminus \{v_1,v_2\})$ and add it to $Z$.
Strongly deleting the vertices of $Z$ will delete all the edges of $F$ and thus separate $v_1$ and $v_2$. Therefore, if such vertices $v_1$ and $v_2$ exist, then we have a strong vertex cut of size at most $|F|$.
Otherwise, $F$ contains the edge $\{v_1,v_2\}$ for all pairs of vertices $v_1 \in V_1$ and $v_2 \in V_2$.

Thus $|V| = |V_1| + |V_2|$ and $F$ contains at least $|V_1| \cdot |V_2|$ edges. Since $|V_1|$ and $|V_2|$ are positive integers, $|V| -1 = |V_1| + |V_2| -1 \leq |V_1| \cdot |V_2| \leq |F|$. However, this contradicts the assumption that $|F| < |V| -1$ and so the result holds.
\end{proof}

\subsection{Minimum transversals versus strong vertex connectivity} \label{section_transversals}

A \textbf{transversal} of a hypergraph $H=(V,E)$ is a subset $T \subseteq V$ such that $T$ has nonempty intersection with every nonempty edge of $H$.
The size of a smallest transversal of a hypergraph $H$ is called the \textbf{transversal number} of $H$ and is denoted $\tau(H)$.
For more details on transversals see~\cite{duchet}.

Note that a subset $T \subseteq V$ is a transversal of $H=(V,E)$ if and only if strongly deleting all vertices in $T$ results in a (possibly null) hypergraph with no nonempty edges. Furthermore, if $W$ is the set of vertices that appear in edges of cardinality 1 then $W$ is a subset of every transversal of $H$, and $T$ is a transversal of $H$ if and only if $T \setminus W$ is a transversal of $H \setminus_S W$.
We now
explore the relationship between the transversal number of a hypergraph $H$ and its strong vertex connectivity.

Consider a connected nontrivial hypergraph $H=(V,E)$ with no edges of cardinality~1 and let $V_{\tau}$ and $V_{\kappa_S}$ be a minimum transversal and a minimum strong vertex cut of $H$, respectively. Then $V_{\tau}$ is a smallest set of vertices whose strong deletion results in {all} (possibly zero) pairs of remaining vertices being separated from each other, while $V_{\kappa_S}$ is a smallest set of vertices whose strong deletion results in at least {one} pair of remaining vertices becoming separated. Thus any transversal of $H$ of size at most $|V| -2$ is also a strong vertex cut of $H$. Combining this with the facts that $\kappa_S(H) \leq |V| -1$ (by definition), edges of cardinality 1 do not affect $\kappa_S(H)$ and cannot decrease $\tau(H)$, and $\tau(H)$ being well defined for any $H$ (because $V$ is always a transversal), proves the following lemma.

\begin{lemma}
If $H=(V,E)$ is a nontrivial hypergraph, then $\kappa_S(H) \leq \tau(H)$.
\end{lemma}

Even though $\tau(H)$ is only an upper bound on $\kappa_S(H)$,
minimum transversals can be used to calculate the strong vertex connectivity of a hypergraph.
Given two distinct vertices $u$ and $v$ in a nontrivial hypergraph $H=(V,E)$, let $H'_{u,v}$ be the hypergraph with $V(H'_{u,v})=V \setminus \{u,v\}$
and
\[E(H'_{u,v}) = \big\{ supp\big(\bigcup_{e \in E(P)} e \big) \setminus \{u,v\}  \,\big|\, P \mbox{ is a } (u,v)\mbox{-path in } H \big\}\]
so the edges of $H'_{u,v}$ are the sets comprised of the vertices, other than $u$ and $v$, that are in the edges of the individual $(u,v)$-paths in $H$.
By the definition of $\kappa_S(H,u,v)$, we have that $\kappa_S(H) = \min \big\{\kappa_S(H,u,v) \,\big|\, {u,v \in V} \big\}$. Note that $\kappa_S(H,u,v) = |V| -1$ when $\{u,v\} \in E$ and that $\kappa_S(H,u,v)$ is equal to the size of a minimum transversal in $H'_{u,v}$ when $\{u,v\} \not \in E$.
However, to rely upon this approach as a means of computing $\kappa_S(H)$ would require finding
$supp\big(\bigcup_{e \in E(P)} e \big)$ for each $(u,v)$-path $P$ in $H$,
which may not be practical.

Later in this paper we will consider the complexity of calculating $\kappa_S(H)$ for various classes of hypergraphs. As we shall see in Section~\ref{section_strong_vs_trans},
although there are similarities between transversals and strong vertex cuts, there exist classes of hypergraphs for which calculating $\tau(H)$ is NP-hard while calculating $\kappa_S(H)$ is in P, and there exist other classes of hypergraphs for which calculating $\tau(H)$ is in P while calculating $\kappa_S(H)$ is NP-hard.

\subsection{Complexity of weak vertex connectivity}
Bahmanian and {\v S}ajna showed in~\cite{bahmanian}
that weak cut vertices of hypergraphs can be found by looking for cut vertices in their associated incidence graphs, which in turn
can be found in polynomial time (see~\cite[page~42]{NagamochiIbaraki}).
Thus, the problem of finding weak cut vertices of hypergraphs is in P. We shall now show that finding minimum weak vertex cuts of hypergraphs is also in P.
We begin with a lemma, the proof of which is an obvious consequence of vertices being adjacent in a hypergraph if and only if they are also adjacent
in the hypergraph's 2-section.

\begin{lemma}\label{lemma_connected_2section}
Let $H=(V,E)$ be a hypergraph and let $u,v \in V$.
Then the vertices $u$ and $v$ are connected in $H$ if and only if they are connected in the 2-section $[H]_2$.
\end{lemma}

We now show that a hypergraph's weak vertex connectivity is the same as the connectivity of its 2-section graph.

\begin{lemma} \label{lemma_weak_cuts_cut_2section}
Let $H=(V,E)$ be a hypergraph and $X \subseteq V$.
Then $X$ is a weak vertex cut of $H$ if and only if $X$ is a vertex cut of $[H]_2$.
Therefore $\kappa_W(H) = \kappa([H]_2)$.
\end{lemma}

\begin{proof}
Take any $X \subset V$.
Observe that $[H \setminus_W X]_2 = [H]_2 \setminus X$, and thus by Lemma~\ref{lemma_connected_2section},
$c(H \setminus_W X) = c([H \setminus_W X]_2) = c([H]_2 \setminus X)$.
Hence $X$ is a weak vertex cut of $H$ if and only if $X$ is a vertex cut of $[H]_2$, and $\kappa_W(H) = \kappa([H]_2)$.
\end{proof}

\begin{thm} \label{thm_kappa_w_poly}
Determining $\kappa_W(H)$ for a hypergraph $H$ and finding a minimum weak vertex cut of $H$ are in P.
\end{thm}

\begin{proof}
By Lemma~\ref{lemma_weak_cuts_cut_2section}, weak vertex cuts in a hypergraph $H$ correspond to vertex cuts in $[H]_2$, and the problem of finding a minimum vertex cut in a graph is in P
(see~\cite[page~42]{NagamochiIbaraki} for a polynomial time algorithm).
\end{proof}

We now demonstrate that
the problem of finding a minimum weak vertex cut of a hypergraph $H= (V,E)$ can be solved, without considering the 2-section of $H$,
by reducing the problem to finding cuts of minimum capacity in $\binom{|V|}{2}$
weighted directed graphs. The latter problem, as well as the conversion, are polynomial (see ~\cite{EdmondsKarp}).

Given a subset $X$ of vertices of a directed graph $G=(V_G,E_G)$,
we let $\partial(X) = \big\{ (u,v) \in E_G \,\big|\, u \in X, v \in V_G \setminus X \big\}$
denote the subset of all edges of $E_G$ having tails in $X$ and heads not in $X$.
If $u \in X$ and $v \in V_G \setminus X$, then $\partial(X)$ is called a \textbf{$\mathbf{(u,v)}$-cut}.
If each edge $e \in E_G$ is weighted with a capacity $cap(e)$, then the \textbf{capacity} of a cut $C$, denoted $cap(C)$, is $\sum_{e \in C} cap(e)$.

Suppose now that we have a nontrivial hypergraph $H=(V,E)$
such that $V = \{v_1,v_2, \dots, v_n \}$ and $E = [ e_1,e_2, \dots, e_m ]$.  Then
\[ \kappa_W(H) = \min \big\{ \kappa_W(H,v_i, v_j)   \,\big|\,   i,j \in \{1,2,\dots, n\}, i \neq j       \big\}\]
and thus calculating $\kappa_W(H)$ can be reduced to determining  $\kappa_W(H,v_i, v_j)$ for all $\binom{|V|}{2}$ pairs, $\{v_i, v_j\}$, of distinct vertices.
Without loss of generality we consider $i=1$ and $j=n>2 $ (if $n=2$, then there are no weak $(v_1,v_2)$-vertex cuts).

Begin by forming the directed graph $G=(V_G,E_G)$ with
\begin{eqnarray*}
V_G &=& \{v_{1,out},v_{2,in}, v_{2,out}, v_{3,in},  v_{3,out}, \dots,  v_{n-1,in},  v_{n-1,out}, v_{n,in}\} \cup E \mbox{, and} \\
E_G &=& \big\{(v_{i,out},e) \,\big|\,  e \in E, v_i \in e, i \neq n\big\} \cup \big\{(e, v_{i,in}) \,\big|\, e \in E, v_i \in e, i \neq 1\big\} \cup \\
  & & \big\{(v_{i,in},v_{i,out}) \,\big|\, i \in \{2,3, \dots,n-1\}\big\}.
\end{eqnarray*}
The formation of $G$ from $H$ can be thought of as follows: for each vertex $v$ of $H$, create two vertices $v_{in}$ and $v_{out}$, and a directed edge from $v_{in}$ to $v_{out}$;  for each edge $e$ in $H$ create a vertex $e$ in $G$; for each edge $\{v,e\}$ in the incidence graph of $H$ create the two directed edges $(e,v_{in})$ and $(v_{out},e)$; and lastly remove $v_{1,in}$ and $v_{n,out}$.


Put a capacity of 1 on all edges of the form $(v_{i,in},v_{i,out})$ and an infinite capacity on all other edges of $G$.
Standard maximum-flow minimum-cut algorithms
such as the Edmonds-Karp algorithm~\cite{EdmondsKarp}
%
%
can then be used to find a minimum capacity $(v_{1,out},v_{n,in})$-cut in $G$ in polynomial time.
The following lemma implies that such a
$(v_{1,out},v_{n,in})$-cut of minimum capacity in $G$
either produces a minimum weak $(v_1,v_n)$-vertex cut in $H$ or
else there are no weak $(v_1,v_n)$-vertex cuts in $H$.


\begin{lemma}
Let $C_G \subseteq V(G)$ be such that $\partial(C_G)$ is a $(v_{1,out},v_{n,in})$-cut of minimum capacity of $G$.
If $cap(\partial(C_G))$ is infinite, then $\kappa_W(H,v_1, v_n) = |V| -1$.
Otherwise $cap(\partial(C_G)) = \kappa_W(H,v_1, v_n) \leq |V| -2$ and the set $C_H = \{v_i \,|\, v_{i,in} \in C_G, v_{i,out}\not \in C_G\}$ is a minimum weak $(v_1,v_n)$-vertex cut in $H$.
\end{lemma}

\begin{proof}
Suppose that the capacity of $\partial(C_G)$ is infinite; then there exists an edge in $H$ that contains both $v_1$ and $v_n$.
Such an edge must exist, otherwise $\partial\big( \{v_{1,out}\} \cup \{v_{i,in} \,|\, 2 \leq i \leq n-1\} \cup \{e \in E \,|\, v_1 \in e\} \big)$ is a $(v_{1,out},v_{n,in})$-cut of (finite) capacity $n-2$.
Since $v_1$ and $v_n$ are adjacent in $H$, there is no weak vertex cut for
which the deletion of the cut results in $v_1$ and $v_n$ becoming
separated, and so $\kappa_W(H,v_1,v_n) = |V| - 1$.

Now suppose that the capacity of $\partial(C_G)$ is finite.
By the construction of $G$, there must be exactly $cap(\partial(C_G))$ indices $i \in \{2, \dots, n-1\}$ such that $v_{i,in} \in C_G$ and $v_{i,out} \not \in C_G$ (with each such $i$ yielding a single edge of capacity 1 in $\partial(C_G)$). Let $I_H$ be this set of $cap(\partial(C_G))$ indices and let $C_H = \{ v_i \,|\, i \in I_H\}$.

We now show that $C_H$ is a weak $(v_1,v_n)$-vertex cut in $H$. Suppose that
\[ P_H = v_1, e_{p_1}, v_{p_1}, e_{p_2}, v_{p_2},\dots, e_{p_b}, v_{p_b} \]
is a $(v_1,v_n)$-path in $H\setminus_W C_H$ (so necessarily $v_{p_b} = v_{n}$).
The existence of $P_H$ implies the existence of the path
\[P_G =  v_{1,out}, e_{p_1}, v_{p_1,in}, v_{p_1,out}, e_{p_2}, v_{p_2,in}, v_{p_2,out},\dots, e_{p_b}, v_{p_b,in}\]
in $G$ (note that the $e_{p_j}$ are edges in $P_H$ but vertices in $P_G$).
As $\partial(C_G)$ is a $(v_{1,out},v_{n,in})$-cut in $G$, the path $P_G$ must contain at least one edge of $\partial(C_G)$.
Let $f$ be such an edge of $E(P_G) \cap \partial(C_G)$.
Since $cap(\partial(C_G))$ is finite, $f = (v_{p_j,in}, v_{p_j,out})$ for some $p_j \in I_H$.
However, this contradicts the fact that $P_H$ is a path in $H\setminus_W C_H$.
Therefore $C_H$ is a weak $(v_1,v_n)$-vertex cut in $H$,
and thus $\kappa_W(H,v_1,v_n) \leq |I_H| = cap(\partial(C_G))$.

We now turn our attention to showing that $cap(\partial(C_G)) \leq \kappa_W(H,v_1,v_n)$. Suppose $C_M \subset V$ is a minimum weak $(v_{1},v_{n})$-vertex cut in $H$.
Let $F = \{(v_{i,in},v_{i,out}) \,|\, v_i \in C_M\} \subset E_G$, observe that $|F| \leq |V|-2$, and consider $G\setminus F$.
Any $(v_{1,out},v_{n,in})$-path in $G\setminus F$ must be of the form
\[Q_G =  v_{1,out}, e_{q_1}, v_{q_1,in}, v_{q_1,out}, e_{q_2}, v_{q_2,in}, v_{q_2,out},\dots, e_{q_\ell}, v_{q_\ell,in} \]
for which there is a corresponding $(v_1,v_n)$-path
\[Q_H = v_1, e_{q_1}, v_{q_1}, e_{q_2}, v_{q_2},\dots, e_{q_\ell}, v_{p_\ell} \]
in $H\setminus_W C_M$. As there are no  $(v_1,v_n)$-paths in $H\setminus_W C_M$, there can be no such  $(v_{1,out},v_{n,in})$-path $Q_G$ in $G\setminus F$.
Hence $v_{1,out}$ and $v_{n,in}$ are separated in $G \setminus F$.
Define
\[
C_F = \{v \in V(G) \,|\, \exists \mbox{ a } (v_{1,out},v) \mbox{-path in } G\setminus F\}.
\]
Since $v_{1,out} \in C_F$ and $v_{n,in} \not \in C_F$, it follows that $\partial(C_F)$ is a $(v_{1,out},v_{n,in})$-cut in $G$
and thus $cap(\partial(C_G)) \leq cap(\partial(C_F))$.
By construction, $\partial(C_F) \subseteq F$ and so
$ cap(\partial(C_G)) \leq cap(\partial(C_F)) \leq |F| = \kappa_W(H,v_1,v_2)$.

It now follows that when $cap(\partial(C_G))$ is finite, it must be that
$ cap(\partial(C_G)) = \kappa_W(H,v_1,v_2) \leq |V|-2$ and, moreover,
that $C_H$ is a minimum weak $(v_1,v_n)$-vertex cut in $H$.
\end{proof}

When creating $G$ from $H$, the vertices $v_2, \dots, v_{n-1}$ of $H$ were split into $in$ and $out$ vertices. Suppose that instead of splitting the vertices of $H$ we split the edges of $H$ into $in$ and $out$ vertices in $G$ and again put capacities of 1 on all $(in,out)$ edges of $G$. Now the minimum capacity $(v_1,v_n)$-cuts of $G$ will correspond to minimum weak $(v_1,v_n)$-disconnecting sets of $H$.  Thus the same technique can be used to find minimum weak disconnecting sets of $H$.

\subsection{Complexity of strong vertex connectivity}
\label{strong_v_con_complexity}

In this section we will show that determining the strong vertex connectivity of an arbitrary hypergraph is NP-hard.
We then consider the complexity of determining the strong vertex connectivity for particular classes of hypergraphs. However, first we present a class of hypergraphs for which the problem is in P.

\begin{lemma} \label{lemma_kappa_s_size_2}
For a hypergraph $H$ with maximum edge size at most 2, the problems of determining $\kappa_S(H)$ and finding a minimum strong vertex cut are in P.
\end{lemma}

\begin{proof}
For hypergraphs of maximum edge size at most 2, the only difference between weakly and strongly deleting a set of vertices is the potential creation of edges of size less than 2.  However, edges of size less than 2 have no effect on the vertex connectivity (weak or strong) of a hypergraph. Thus, for hypergraphs of this type, the strong vertex cuts are exactly the weak vertex cuts and the result follows from Theorem~\ref{thm_kappa_w_poly}.
\end{proof}

\begin{thm} \label{thm_kappa_strong}
The problem of determining $\kappa_S(H)$ is NP-hard for arbitrary hypergraphs. Furthermore, the problem remains NP-hard when $H$ is restricted to hypergraphs with maximum edge size at most 3.
\end{thm}

\begin{proof}
The problem of finding the size of a minimum vertex cover in a simple graph is known to be NP-hard~\cite{karp}.
We will reduce this NP-hard problem to the problem of finding  $\kappa_S(H)$ for a hypergraph $H$ with maximum edge size of 3.

Let $G = (V_G, E_G)$ be a simple graph with $V_G = \{ x_1, x_2, \dots, x_n\}$ and at least one edge. The set $V_G \setminus \{x_1\}$ intersects every edge in $E_G$ and is thus a vertex cover of $G$. Therefore the size of a minimum vertex cover of $G$ is at most $n-1$.  Let
\begin{eqnarray*}
A_u &=& \{u_1, \dots, u_n\}, \\
A_v &=& \{v_1, \dots, v_n\}, \\
V   &=& V_G \cup A_u \cup A_v
\end{eqnarray*}
such that $A_u$, $A_v$ and $V_G$ are pairwise disjoint, and let
\begin{eqnarray*}
E   &=& \big\{ \{u_i,x_j\} \,\big|\, 1 \leq i,j \leq n\big\} \cup \big\{ e \cup \{v_i\} \,\big|\, e \in E_G, 1 \leq i \leq n\big\}.
\end{eqnarray*}
Consider the hypergraph $H = (V,E)$.
Given an edge $e$ in $G$, we shall denote the corresponding edge $ e \cup \{v_i\}$ in $H$ by $e_{v_i}$. Note that it follows from $|E_G| \geq 1$ that $H$ is connected.

We now make the following four claims regarding vertex covers of $G$, strong vertex cuts in $H$, and the relationship between them.
\\

\noindent \textbf{Claim 1:} The vertex covers of $G$ are exactly the strong $(A_u,A_v)$-vertex cuts in $H$. \\

\noindent \textbf{Claim 2:} If $C$ is a minimum strong vertex cut of $H$, then $C \cap  A_v = \emptyset$. \\

\noindent \textbf{Claim 3:} If $C$ is a minimum strong vertex cut of $H$, then $C \cap A_u = \emptyset$. \\

\noindent \textbf{Claim 4:} If $C$ is a minimum strong vertex cut of $H$, then $C$ separates $A_u$ and $A_v$ (i.e., $C$ is a minimum strong $(A_u,A_v)$-vertex cut in $H$). \\

Claims 2 and 3 will be used in the proof of Claim~4, while the reduction of vertex covers in simple graphs to strong vertex cuts in hypergraphs of maximum edge size at most 3 will follow directly from the construction of $H$ from $G$ and Claims~1 and~4. Therefore Theorem~\ref{thm_kappa_strong} holds subject to validating these four claims.
\\

\noindent \textbf{Proof of Claim 1:}\\
Suppose that $C$ is a strong $(A_u, A_v)$-vertex cut of $H$. Then $C \subseteq V_G$ and there exist no $(u_i,v_j)$-paths in $ H\setminus_S C$.
If $e = (x_k, x_{\ell})$ is an edge of $G$, then $u_i, \{u_i, x_k\}, x_k, e_{v_j}, v_j$ is a $(u_i,v_j)$-path in $H$ for each $1 \leq i,j \leq n$. Since $C$ is a strong $(A_u, A_v)$-vertex cut of $H$, this path does not exist in $H \setminus_S C$.  As $u_i,v_j \not \in C$, the only way for this path not to exist in $H \setminus_S C$ is if $\{x_k,x_{\ell}\} \cap C \neq \emptyset$.  Therefore $C$ is a vertex cover of $G$.

Conversely, suppose that $C \subseteq V_G$ is a vertex cover of $G$. Then $C$ intersects all edges of $G$ which implies that
$deg_{H \setminus_S C} (v_j) = 0$ for each $j \in \{1,2,\ldots,n\}$.
Therefore $C$ is a strong $(A_u,A_v)$-vertex cut in $H$ and Claim~1 holds.
\\

\noindent \textbf{Proof of Claim 2:}\\
Let $C$ be a minimum strong vertex cut of $H$ and suppose that $C \cap A_v \neq \emptyset$. Without loss of generality we may assume that $v_1 \in C$.

By the minimality of $C$ we know that $C' = C \setminus v_1$ is not a strong vertex cut of $H$ and that $v_1$ is a strong cut vertex of $H' = H \setminus_S C'$. This implies that there exist two distinct vertices $w,z$ in $H$ that are adjacent to $v_1$ in $H'$ but separated from each other in $H \setminus_S C$. Namely, consider a path in $H'$ between vertices that are separated in  $H \setminus_S C = H' \setminus v_1$. Either $v_1$ is an internal vertex of the path or $v_1$ is not on the path but is contained in an edge of the path. Both of these cases lead to the desired $w,z$ pair.

Suppose that $\{v_1,w,z\}$ is an edge of $H'$. Then $\{v_j,w,z\}$ is an edge of $H \setminus_S C$
for each $v_j \in A_v \setminus C$.
This would contradict $w$ and $z$ being separated in $H \setminus_S C$ unless $v_j \in C$ for each $j \in \{1,2,\ldots,n\}$.
However, $v_j \in C$ for each $j \in \{1,2,\ldots,n\}$ contradicts
the fact that $C$ is a minimum strong vertex cut of $H$, because $V_G \setminus \{x_1\}$ is a strong vertex cut of $H$ of size $n-1$.

Since
there is no edge $\{v_1,w,z\}$ in $H'$, the construction of $H$ implies that there exist two edges $\{v_1,w, y_w \}$ and $\{v_1,z,y_z\}$ in $H'$. But then \[ P = w, \{v_j,w,y_w\}, v_j, \{v_j, z, y_z\}, z\] would be a $(w,z)$-path in $H \setminus_S C$ for any $v_j \in A_v \setminus C$. This would again imply that $A_v \subseteq C$ which would contradict the fact that $C$ is a minimum strong vertex cut of $H$. We thus conclude that $C \cap A_v = \emptyset$ and Claim~2 holds.
\\

\noindent \textbf{Proof of Claim 3:}\\
%
Let $C$ be a minimum strong vertex cut of $H$ and suppose that $C \cap A_u \neq \emptyset$. Without loss of generality, we may assume that $u_1 \in C$.

By the minimality of $C$ we know that $C' = C \setminus u_1$ is not a strong vertex cut of $H$ and that $u_1$ is a strong cut vertex of $H' = H \setminus_S C'$. This implies that there exist two vertices $w,z$ in $H$ that are adjacent to $u_1$ in $H'$ but separated from each other in $H \setminus_S C$.
As $u_1$ is adjacent only to vertices of $V_G$, clearly $w,z \in V_G$.
Since $\{u_i,x_j\}$ is an edge of $H$ for each $1 \leq i,j \leq n$,
the only way for two vertices of $V_G \setminus C$ to be separated in $H \setminus_S C$ is if $A_u \subseteq C$,
which would contradict the fact that $C$ is a minimum strong vertex cut of $H$. Therefore $C \cap A_u = \emptyset$ and Claim~3 holds.
\\

\noindent \textbf{Proof of Claim 4:}\\
Let $C$ be a minimum strong vertex cut of $H$. Claims~2 and~3 imply that $C \subseteq V_G$ and thus it only remains to show that for all $i,j$ there is no $(u_i,v_j)$-path in $H \setminus_S C$.

Since $C$ is a minimum strong vertex cut in $H$, $|C| \leq n-1$ and so there exists an $i \in \{1,2,\ldots,n\}$ such that $x_i \in V_G \setminus C$.
Each vertex of $A_u$ is adjacent in $H$ to each vertex of $V_G$,
and thus the vertices of $A_u \cup (V_G \setminus C)$ are all in the same connected component of $H \setminus_S C$.

Suppose that for some $j \in \{1,2,\ldots,n\}$ the vertex $v_j$ is in the same connected component of $H \setminus_S C$ as the vertices of $A_u \cup (V_G \setminus C)$.
Then there exists a $(u_1, v_j)$-path in $H \setminus_S C$. Let $e_{v_j}$ be the last edge in such a path and let $z \in e_{v_j} \setminus \{v_j\}$.

Since $e_{v_j}$ is an edge of $H \setminus_S C$ and $C \cap  A_v = \emptyset$, we have that $e_{v_k}$ is in $H \setminus_S C$ for all $1 \leq k \leq n$. But then $v_j, e_{v_j}, z, e_{v_k}, v_k$ is a $(v_j,v_k)$-path in $H \setminus_S C$ for all $k \neq j$. This implies that each vertex of $A_v$ is in the same connected component of $H \setminus_S C$ as $v_j$, which contradicts the fact that $H \setminus_S C$ is disconnected. Therefore, no vertex of $A_v$ is in the same connected component of $H \setminus_S C$ as the vertices of $A_u \cup (V_G \setminus C)$ and Claim~4 holds.
\\

The construction of $H$ from $G$, together with Claims~1 and~4, shows a polynomial reduction of the problem of finding minimum vertex covers in simple graphs to the problem of finding minimum strong vertex cuts in hypergraphs with maximum edge size at most 3.
Theorem~\ref{thm_kappa_strong} is therefore proved.
\end{proof}

Next we consider the following particular classes of hypergraphs.

\begin{itemize}
\item A hypergraph $H$ is \textbf{bicolourable} if each vertex of $V$ can be labelled with one of two colours so that each edge of cardinality at least two is not monochromatic.

\item A hypergraph $H=(V,E)$ is said to have the \textbf{Helly} property if, for every multiset $F \subseteq E$ of pairwise intersecting edges,
there is a vertex $u_F \in V$ that is incident with each edge of $F$.

\item Any hypergraph $H$ for which $\tau(H)$ equals the size of a maximum matching is said to have the \textbf{K\"{o}nig property}.

\item A hypergraph is \textbf{normal} if each of its strong subhypergraphs has the K\"{o}nig property.

\item A hypergraph $H$ is \textbf{arboreal} if there exists a (not necessarily unique) tree $T$ on the same vertex set as $H$ such that each edge of $H$ induces a connected subgraph (i.e., subtree) of $T$. Such a tree $T$ is called a \textbf{representative tree} of $H$.

\item A hypergraph is \textbf{totally balanced} if, for every cycle $C$ of length at least 3, there is an edge of $C$
that contains at least three vertices of $C$.

\item A hypergraph is an \textbf{interval hypergraph} if there exists a total ordering of the vertices such that all edges are intervals of the ordering.

\end{itemize}

In~\cite{duchet} Duchet discusses these and other classes of hypergraphs.  Furthermore, a hierarchy is presented, showing that interval hypergraphs
are totally balanced, that totally balanced hypergraphs are arboreal, that arboreal hypergraphs are normal, and that normal hypergraphs not only have
the Helly and K\"{o}nig properties but are also bicolourable.
These classes of hypergraphs are also discussed in~\cite{Bretto2013}.

Recall the hypergraph $H$ that is constructed from the graph $G$ in the proof of Theorem~\ref{thm_kappa_strong}. Every edge of $H$ has one vertex in $A_u \cup A_v$ and at least one vertex in $V_G$.  Partitioning the vertices of $H$ in this way induces a 2-colouring of $H$ that establishes that $H$ is bicolourable.
Moreover, every edge of $H$ intersects $V_G$ and so the size of a minimum transversal of $H$ is at most $|V_G| = n$.
Now, consider the fact that if $M$ is any matching
and $T$ is any transversal in a hypergraph, then $|M| \leq |T|$ because $T$ contains at least one vertex from each edge in $M$.
Since $\{ (u_i,x_i) \,|\, 1 \leq i \leq n\}$ is a matching of size $n$ in $H$,
and $H$ has a transversal of size $n$, we have that the size of a maximum matching equals the size of a minimum transversal for $H$
(i.e., $H$ has the {K\"{o}nig property}). Thus we have the following two corollaries.

\begin{cor}
The problem of determining $\kappa_S(H)$ is NP-hard for bicolourable hypergraphs.
\end{cor}

\begin{cor} \label{thm_kappa_s_konig}
The problem of determining $\kappa_S(H)$ is NP-hard for hypergraphs with the K\"{o}nig property.
\end{cor}

We have shown that
determining $\kappa_S(H)$ for general hypergraphs, hypergraphs with maximum edge size at most 3, bicolourable hypergraphs and hypergraphs with the K\"{o}nig property is NP-hard.
However, Lemma~\ref{lemma_kappa_s_size_2} establishes that there are classes of hypergraphs (such as those having maximum edge size 2) for which the problem is in P. Below we give other classes of hypergraphs for which the problem is in P.

\begin{lemma} \label{lemma_kappa_s_p_arboreal}
Suppose $H$ is an arboreal hypergraph. If $H$ is connected then $\kappa_S(H) = 1$; otherwise $\kappa_S(H) =0$.
Furthermore,  the problems of determining $\kappa_S(H)$ and finding a minimum strong vertex cut when one exists are in P.
\end{lemma}

\begin{proof}
Let $H=(V,E)$ be an arboreal hypergraph and let $T = (V,F)$ be a representative tree of $H$.
As noted in~\cite{BBF2005}, representative trees of arboreal hypergraphs can be found in polynomial time by adapting an inductive proof of Slater~\cite{Slater1978} (also see~\cite[page~64]{NagamochiIbaraki}). Determining if $H$ is connected is also in P
(as this can be determined by a breadth first search).

If $H$ is null or trivial, then $\kappa_S(H)=1$. If $H$ is disconnected, then $\kappa_S(H)=0$. If $|V| = 2$ and $H$ is connected, then $H$ has no strong (or weak) vertex cuts and $\kappa_S(H) = |V| -1 = 1$. Thus, henceforth we assume that $H$ is connected and $|V| \geq 3$.

As $|V| \geq 3$, $T$ has a vertex that is not a leaf.  Let $u \in V$ be such a vertex, i.e., $deg_T(u)>1$.
Clearly $u$ can be found in polynomial time,
and since $u$ is not a leaf, $T \setminus u$ is disconnected. Let $v_1$ and $v_2$ be vertices in different connected components of $T \setminus u$. Suppose that there exists a path $P$ from $v_1$ to $v_2$ in $H \setminus_S u$. Then $P$ would contain an edge $e$ in $H \setminus_S u$ that contains vertices from two different connected components of $T \setminus_S u$. However, since $u \not \in e$ this would mean that $e$ does not induce a connected subgraph of $T$ and hence contradicts the fact that $H$ is arboreal. Therefore $u$ is a strong cut vertex of $H$ and $\kappa_S(H) = 1$.
\end{proof}

\begin{cor}
Let $H$ be a hypergraph that is either:
\begin{enumerate}
\item totally balanced; 
\item without cycles of length at least 3; or
\item an interval hypergraph. 
\end{enumerate}
If $H$ is connected then $\kappa_S(H) = 1$; otherwise $\kappa_S(H) =0$. Furthermore, for such hypergraphs
the problems of determining $\kappa_S(H)$ and finding a minimum strong vertex cut when one exists are in P.
\end{cor}

\begin{proof}
Interval hypergraphs and hypergraphs without cycles of length at least 3 are totally balanced, and totally balanced hypergraphs are arboreal (see \cite[page~395]{duchet}).
\end{proof}

We have seen that finding minimum strong vertex cuts for arboreal hypergraphs is in P,
yet the problem is NP-hard for bicolourable hypergraphs and for hypergraphs with the K\"{o}nig property.
An interesting question, for which we do not yet have an answer, is what is the complexity of this problem for normal hypergraphs
(which contain arboreal hypergraphs as a subclass, and are contained as a subclass of both bicolourable hypergraphs and hypergraphs with the K\"{o}nig property)?

\subsection{Complexity of finding strong vertex cuts versus transversals} \label{section_strong_vs_trans}
Similarities between transversals and strong vertex cuts were discussed in Section~\ref{section_transversals} where it was shown that $\kappa_S(H) \leq \tau(H)$ for any nontrivial hypergraph $H$. We shall now consider the relative complexities of finding minimum transversals and minimum strong vertex cuts in hypergraphs.

Finding a minimum transversal of an arbitrary hypergraph is an NP-hard problem
(it is equivalent to the Set Covering problem of~\cite{karp}).
Theorem~\ref{thm_kappa_strong} showed that finding a minimum strong vertex cut in an arbitrary hypergraph is also NP-hard. Given the similarity of the two problems and their complexity over the set of all hypergraphs, it is natural to compare their computational complexities over various classes of hypergraphs.  For example, over which classes of hypergraphs do the two problems have the same complexity, over which classes do they differ and when they differ is one problem consistently harder than the other?

Consider the class of hypergraphs with the K\"{o}nig property. Corollary~\ref{thm_kappa_s_konig} states that finding minimum strong vertex cuts is NP-hard for hypergraphs with the K\"{o}nig property.
The maximum fractional matching problem and the minimum fractional transversal problem are dual linear programs (see \cite[Chapter~7, Section~3.2]{duchet}
for more details). Denote the optimal values to these problems by $\alpha^{*}(H)$ and $\tau^{*}(H)$, respectively.
As they are dual linear programs,  these problems can be solved in polynomial time and additionally
\[
\alpha(H) \leq \alpha^{*}(H) = \tau^{*}(H)  \leq \tau(H),
\]
where $\alpha(H)$ is the maximum size of a matching in $H$.
As $\alpha(H) = \tau(H)$ for hypergraphs with the K\"{o}nig property, it follows that calculating $\tau(H)$ for hypergraphs with the K\"{o}nig property is in P.
Here we have a class of hypergraphs for which calculating $\tau(H)$ is in P, but calculating $\kappa_S(H)$ is NP-hard.

As arboreal hypergraphs have the K\"{o}nig property, calculating $\tau(H)$ for arboreal hypergraphs is in P.  Lemma~\ref{lemma_kappa_s_p_arboreal} indicates that determining $\kappa_S(H)$ is in P for arboreal hypergraphs, thus we have that for this class of hypergraphs the calculations of $\tau(H)$ and $\kappa_S(H)$ are both in P.

Let $H=(V,E)$ be a hypergraph and define $\ddot{H}$ to be the hypergraph with vertex set $V \cup \{u_1,u_2\}$ and edge set $E \cup \big\{ \{u_1,u_2\}, \{u_1,u_2\} \cup V \big\}$, where $u_1 \neq u_2$ and $u_1,u_2 \not \in V$. Let $\frak{H} = \{\ddot{H} \,|\, H \mbox{ is a hypergraph}\}$.  For any hypergraph $H$ with at least one vertex, $\{u_1\}$ is a minimum strong vertex cut of $\ddot{H}$ and $\tau(\ddot{H}) = \tau(H) +1$ (the minimum transversals of $\ddot{H}$ are exactly the minimum transversals of $H$ with $u_1$ or $u_2$ added). Therefore,
over $\frak{H}$ the problem of finding a minimum strong vertex cut is in P, while the problem of finding a minimum transversal is NP-hard.

\section{Concluding remark}
Having established with Theorem~\ref{thm_kappa_strong} that the problem of determining strong vertex connectivity is NP-hard for arbitrary hypergraphs, it is natural to ask what the best running time is for solving this problem.  We leave this as an open question.


\bibliographystyle{plain}
\bibliography{hypergraph_refs}
\end{document}